\documentclass[12pt]{article}

\usepackage{amssymb}
\usepackage{amsfonts}
\usepackage{amsmath}
\usepackage{amsthm}
\usepackage{graphicx}
\usepackage{caption}
\usepackage[usenames, dvipsnames]{xcolor}
\usepackage{verbatim}
\usepackage{dsfont}
\usepackage{color}

\newtheorem{theorem}{Theorem}[section]
\newtheorem{lemma}[theorem]{Lemma}
\newtheorem{remark}[theorem]{Remark}
\newtheorem{example}[theorem]{Example}
\newtheorem{proposition}[theorem]{Proposition}
\newtheorem{definition}[theorem]{Definition}

\newtheorem{corollary}[theorem]{Corollary}
\newtheorem{fig}[theorem]{Figure}

\newcommand{\bthe}{\begin{theorem}}
\newcommand{\ethe}{\end{theorem}}

\newcommand{\ben}{\begin{enumerate}}
\newcommand{\een}{\end{enumerate}}

\newcommand{\bit}{\begin{itemize}}
\newcommand{\eit}{\end{itemize}}

\newcommand{\beq}{\begin{equation}}
\newcommand{\eeq}{\end{equation}}

\newcommand{\ble}{\begin{lemma}}
\newcommand{\ele}{\end{lemma}}

\newcommand{\bde}{\begin{definition}\rm}
\newcommand{\ede}{\halmos\end{definition}}

\newcommand{\bco}{\begin{corollary}}
\newcommand{\eco}{\end{corollary}}

\newcommand{\bpr}{\begin{proposition}}
\newcommand{\epr}{\end{proposition}}

\newcommand{\brem}{\begin{remark}\rm}
\newcommand{\erem}{\end{remark}}

\newcommand{\bproof}{\begin{proof}}
\newcommand{\eproof}{\end{proof}}

\newcommand{\bexam}{\begin{example}\rm}
\newcommand{\eexam}{\end{example}}

\newcommand{\bfi}{\begin{fig}}
\newcommand{\efi}{\end{fig}}

\newcommand{\btab}{\begin{tab}}
\newcommand{\etab}{\end{tab}}

\newcommand{\beao}{\begin{eqnarray*}}
\newcommand{\eeao}{\end{eqnarray*}\noindent}

\newcommand{\beam}{\begin{eqnarray}}
\newcommand{\eeam}{\end{eqnarray}\noindent}

\newcommand{\barr}{\begin{array}}
\newcommand{\earr}{\end{array}}

\newcommand{\bdis}{\begin{displaymath}}
\newcommand{\edis}{\end{displaymath}\noindent}

\newcommand*\colvec[3][]{
    \begin{pmatrix}\ifx\relax#1\relax\else#1\\\fi#2\\#3\end{pmatrix}
}

\def\RR{{\mathbb R}}

\newcommand{\NN}{\mathbb{N}}

\def\calf{{\mathcal{F}}}

\newcommand{\bbn}{\mathbb{N}}

\newcommand{\bbr}{\mathbb{R}}

\newcommand{\al}{{\alpha}}

\newcommand{\ga}{{\gamma}}

\newcommand{\si}{{\sigma}}
\newcommand{\vp}{\varphi}
\newcommand{\eps}{\varepsilon}

\newcommand{\var}{\text{var}}
\newcommand{\cov}{\text{cov}}

\newcommand{\halmos}{\quad\hfill\mbox{$\Box$}}  

\numberwithin{equation}{section}  

\begin{document}

\title{Asymmetric COGARCH processes} 

\author{Anita Behme\thanks{Center for Mathematical Sciences, Technische Universit\"at M\"unchen,  85748 Garching, Boltzmannstrasse 3, Germany, e-mail: behme@ma.tum.de and cklu@ma.tum.de},
Claudia Kl\"uppelberg$^\ast$\hspace{-5pt}, and
Kathrin Mayr$^\ast$
}

\maketitle

\begin{abstract}
Financial data are as a rule asymmetric, although most econometric models are symmetric.
This applies also to continuous-time models for high-frequency and irregularly spaced data. 
We discuss some asymmetric versions of the continuous-time GARCH model, concentrating then on the GJR-COGARCH.
We calculate higher order moments and extend the first jump approximation.
These results are prerequisites for moment estimation and pseudo maximum likelihood estimation of the GJR-COGARCH parameters, respectively, which we derive in detail.
\end{abstract}

\,\,\,
secondary: \,\,\,62F10; 62M10; 90G70\\


\section{Introduction}\label{s1}

In 1982 Engle \cite{engle82} suggested an autoregressive conditionally heteroskedastic (ARCH) model for the variance of the United Kingdom inflation. 
In this model the conditional variance was modeled as an autoregressive process of past variances. 
Bollerslev~\cite{bollerslev} enriched this model by an additional term of past squared observations resulting in the generalized ARCH (GARCH) model, which is nowadays one of the most prominent econometric models as it captures relevant stylized facts of econometric data. 
It has the form
\begin{align}
Y_n &= \sigma_n\varepsilon_n\,,\quad
\sigma_n^2 = \theta + \sum_{i=1}^{q}{\alpha_{i} Y_{n-i}^2} + \sum_{j=1}^{p}{\beta_{j} \sigma_{n-j}^2},\quad n\in\NN,
\end{align}
for i.i.d. random variables $(\varepsilon_n)_{n\in\bbn}$ with $E\left[\varepsilon_{n}\right]=0$, $\var(\varepsilon_{n})=1$ and $\eps_n$ independent of $\calf_{n-1}$, the sigma algebra generated by $\{Y_{k}: k\le n-1\}$.  
The parameters satisfy $\theta>0$, $\alpha_{i}\geq0$, $\beta_{i}\geq0$ with $\al_q,\beta_p>0$.

However, in real data, there is an asymmetric response of the volatility, called the leverage effect, which says that stock returns are negatively correlated with changes in return volatility. 
More precisely, volatility tends to rise in response to bad news and to fall in response to good news. 
This effect has also been investigated in empirical studies, which show the statistical significance of asymmetry in financial time series models (cf. \cite{dingetal,glostenGJR,hentschel,McKM,PenzerWangYao,rabezakoian}).

As a consequence of their empirical findings, Ding, Granger and Engle~\cite{dingetal} introduced an Asymmetric Power GARCH (APGARCH) model defined as
\begin{align}\label{eq:APARCH1}
Y_n &= \varepsilon_n \sigma_n\,,\quad 
\sigma_n^\delta = \theta + \sum_{i=1}^{q}{\alpha_{i} h(Y_{n-i})} + \sum_{j=1}^{p}{\beta_{j} \sigma_{n-j}^\delta}, \quad n\in\NN, 
\end{align}
for i.i.d. random variables $(\varepsilon_n)_{n\in\bbn}$ with $E\left[\varepsilon_{n}\right]=0$, $\var(\varepsilon_{n})=1$, and $\eps_n$ independent of $\calf_{n-1}$.
The asymmetry is introduced by $h(x)=\left(\left| x \right|- \gamma x\right)^\delta$ with $\delta>0$ and $\left|\gamma \right|<1$;
moreover, $\theta >0$, $\alpha_{i}\geq0$, $\beta_{i}\geq0$  with $\al_q,\beta_p>0$.

It is shown in \cite{dingetal} that the APGARCH model includes several important ARCH and GARCH models as special cases. In particular, if $\delta=2$ the model includes Engle's ARCH($p$) \cite{engle82}, Bollerslev's GARCH($p,q$) \cite{bollerslev} and the GJR model, named after Glosten, Jagannathan and Runkle~\cite{glostenGJR}, while for $\delta=1$ the Threshold GARCH (TARCH) model (\cite{rabezakoian,zakoian}) can be obtained; see \cite{franqzakoian} for further information on GARCH-type models.

With the advent of high-frequency data and irregularly spaced tick-by-tick data, continuous-time models came into the focus of econometrics. Nelson \cite{nelson} derived a continuous-time GARCH model by a diffusion approximation, which yields continuous prices and volatilities driven by Brownian motions.
Consequently, Nelson's GARCH diffusion model cannot model jumps in prices and volatilities. 
However, it retains the heavy (Pareto) tails of the original GARCH model.

At the beginning of the new millennium empirical studies established stylized facts of high-frequency data, giving the important insight that prizes and volatilities exhibit jumps, including common jumps (cf. the excellent monographs \cite{JP12} and \cite{AitSJ} for insight and further references). 
In 2004 Kl\"uppelberg, Lindner and Maller~\cite{KLM:2004} suggested a continuous-time GARCH(1,1) (COGARCH(1,1)) model capturing the jump features of high-frequency data, and they proved properties like strict stationarity and second order behaviour. 
Moment estimation for this COGARCH model works very well for high-frequency data as demonstrated in \cite{zappetal}, where also a simple leverage term has been added. 
Maller et al.~\cite{mallermuellerszimayer} derived a first jump approximation, which provides a sequence of GARCH models converging to the original COGARCH process in probability in the Skorokhod topology.
This allows for the use of existing software for maximum likelihood estimation for GARCH processes and is also applicable to non-equidistantly sampled data.

In this paper we will discuss an asymmetric  COGARCH(1,1) model, which takes care of the observed leverage effect in a systematic way. 
The model is a continuous-time version of \eqref{eq:APARCH1}  with $\delta=2$. 
We define the GJR-COGARCH in Section~\ref{s2}, derive first properties and present some simulation. 
Sections~\ref{s3} and \ref{s4} contain the estimation methods as well as their prerequisites. In particular, in Section~\ref{s3} we calculate the moments of the asymmetric model and apply these to obtain explicit moment estimators.
In Section~\ref{s4} we extend the first jump approximation from \cite{mallermuellerszimayer} to the asymmetric model and prove its convergence. This is then used to derive a pseudo maximum likelihood estimator for the parameters of the GJR-COGARCH.

\section{The GJR-COGARCH}\label{s2}

Recall the GJR-GARCH(1,1) which is defined  as
\begin{align} \label{eq:GJR}
Y_n&:=\sigma_n \varepsilon_n\,,\quad \sigma_n^{2}=\theta + \alpha (|Y_{n-1}|-\gamma Y_{n-1})^2 + \beta \sigma_{n-1}^{2}, \quad n\in\NN, 
\end{align}
for $\theta\geq 0$, $\alpha, \beta>0$, $|\gamma|< 1$ and an i.i.d. noise sequence $(\varepsilon_n)_{n\in\NN}$ with $E[\varepsilon_0]=0$ and $\var(\varepsilon_0)=1$.
Following the construction method of the COGARCH(1,1) in \cite{KLM:2004} and using a reparametrization of the parameters by defining $\eta=-\log\beta$ and $\vp=\alpha/\beta$, a continuous-time GJR-GARCH (GJR-COGARCH) can be defined as follows (cf. \cite{lee10}):
\beam
d G_{t} &=& \sigma_{t} d L_{t}, \quad t\geq 0, \quad G_{0}=0,\label{eq:defCOGJRint}\\
\sigma_{t}^2 &=& \sigma_{0}^2+\theta t -\eta \int_{0}^{t} \sigma_{s}^{2}ds + \vp \sum_{0<s\leq t}{\sigma_{s}^2 h(\Delta L_{s})} , \quad t\geq 0,\quad \sigma_{0}^2\ge 0,\label{eq:defCOGJR}
\eeam
where $h(x)=(|x|-\gamma x)^2$ with $\left| \gamma \right|<1$, and $\theta,\eta,\vp > 0$.
The L\'evy process $L$ has L\'evy measure $\nu_L\neq 0$, independent of $\sigma_0^2$. 
We choose $L$ symmetric so that the asymmetry of the model originates in $\gamma$ only. In particular, throughout this paper we will use $E[L_1]=0$ and $E[L_1^2]=1$.
Note that for a symmetric L\'evy process the sign of the chosen parameter $\gamma$ becomes irrelevant for the resulting process as positive and negative jumps of the same size appear with the same probability. Hence we will assume from now on that $\gamma \in [0,1)$.

\brem
Asymmetry of a COGARCH can, of course, also be achieved by choosing an asymmetric L\'evy process as driving process in the original symmetric COGARCH. 
Replacing in \eqref{eq:defCOGJR} the term $h(\Delta L_s)$ for $L$ with symmetric L\'evy measure $\nu_L$ by $\Delta L_s^2$ with asymmetric L\'evy measure $$\nu_a(dx)=\nu_L(dx) ((1-\ga)\mathds{1}_{\{x\ge0\}}+ (1+\ga)\mathds{1}_{\{x<0\}})$$ yields the same model.
However, we prefer to have the asymmetry as a model parameter which we can estimate by standard statistical procedures.
\erem

The following lemma summarizes some properties of the GJR-COGARCH volatility which we will need later on. Analogous properties of the COGARCH can be found in \cite[Lemma 4.1]{KLM:2004} and \cite[Prop. 3.2]{behmelindnermaller}.

\ble
(a) \, The asymmetric GJR-COGARCH volatility $\left(\sigma_{t}^2\right)_{t\geq0}$ is a generalized Ornstein-Uhlenbeck process with representation
\begin{equation} \label{eq:COGJR}
\sigma_{t}^2=\left(\theta\int_{0}^{t}{e^{X_{s-}}}ds + \sigma_{0}^2\right) e^{-X_{t}}, \quad t\geq 0,
\end{equation}
where $X$ in \eqref{eq:COGJR} is a spectrally negative L\'evy process defined as
\begin{equation}
X_{t} = t\eta - \sum_{0<s\leq t}{\log\left(1+\vp h(\Delta L_{s})\right)}, \quad t\geq 0, \label{eq:coAPARCHX}
\end{equation}
whose Laplace exponent $\Psi(u)=E[e^{-uX_1}]$, $u\geq 0$, is given by
\begin{equation}
\Psi(u)=-\eta u + \int_{\bbr}{\left(\left(1+\vp h(y)\right)^{u}-1\right)\nu_{L}(dy)}, \label{eq:142}
\end{equation}
and it is finite for $u>0$ if and only if $E[L^{2u}]<\infty$.\\[2mm]
(b) \, 
 Provided the quantities are finite, the following identities hold:
\begin{align} \label{eq:psi1}
 \Psi(1)&=-\eta + \vp (1+\gamma^2) \int_{\bbr} y^2 \nu_{L}(dy)= -\eta + \vp (1+\gamma^2) E[L_1^2]\quad \mbox{and} \\
\Psi(2)&= 2 \Psi(1) + \vp^2(1+6\gamma^2+\gamma^4) \int_{\bbr} y^4 \nu_{L}(dy). \label{eq:psi2}
\end{align}
(c) \,  The process $(\sigma_{t}^2)_{t\geq 0}$ is the unique solution of the stochastic differential equation
\begin{equation} \label{eq:SDEcoAPARCH}
d\sigma_{t}^2= \theta dt + \sigma_{t-}^2 dU_t, \quad t> 0,
\end{equation}
with driving L\'evy process
$$U_t=-X_t +\sum_{0<s\leq t} (e^{-\Delta X_s}-1+\Delta X_s)= -\eta t + \vp \sum_{0<s\leq t} h(\Delta L_{s}).$$
\ele

We shall work with the stationary solution of the GJR-COGARCH volatility, whose existence is guaranteed under certain conditions as given in the following proposition.

\bpr
(a) \, The GJR-COGARCH volatility \eqref{eq:COGJR} has a stationary distribution if and only if the integral $\int_{0}^{\infty}{e^{-X_{t-}}}dt$ converges a.s. to a finite random variable. 
This is the case if and only if $E[L_1^2]<\infty$ and
\begin{equation}
\int_{\bbr}{\log \left(1+\vp h(y)\right)\nu_{L}(dy)}<\eta. \label{eq:143}
\end{equation}
(b) \, The stationary distribution of the GJR-COGARCH volatility is uniquely determined by the law of
\begin{equation} \label{eq:coGJRstat}
 \sigma^2_\infty:=\theta \int_{0}^{\infty}{e^{-X_{t-}}}dt.
\end{equation}
(c) \,  Equation \eqref{eq:143} holds if $\Psi(1)\leq 0$. 
Moreover, in this case also the corresponding symmetric COGARCH volatility  (i.e. with $\gamma=0$) has a stationary distribution.
\epr

\bproof
We use the generalized OU representation of $(\sigma_t^2)_{t\geq 0}$ as given in \eqref{eq:COGJR} and \eqref{eq:SDEcoAPARCH}.
(a)  is a consequence of  \cite[Thm. 3.1]{lee10} or \cite[Theorem 2.b]{KLM:2006}, and (b) of \cite[Thm. 2.1]{behmelindnermaller}.
(c) holds true as $\log(1+x)<x$ for positive $x$.  The second part holds, since \eqref{eq:psi1} implies that $\Psi(1)>-\eta +\vp E[L_1]$.
Now apply  \cite[Thm. 2.b]{KLM:2006}.
\eproof

In Figure \ref{fig:simulation}, we depict a simulation of COGARCH and GJR-COGARCH processes, both driven by the same compound Poisson process with rate 1 and standard normal jumps.
Although the sample paths of the symmetric and asymmetric COGARCH in the first row of Figure~\ref{fig:simulation} look similar, the returns in the second row already exhibit more pronounced downwards and less pronounced upwards peaks.
This is due to the volatility process depicted in the third row, where the asymmetry in the jumps has rather dramatic consequences.

\begin{figure}[ht] 
\vspace*{-3mm}
\begin{center}
$$
\barr{cc}
\includegraphics[width = 6cm]{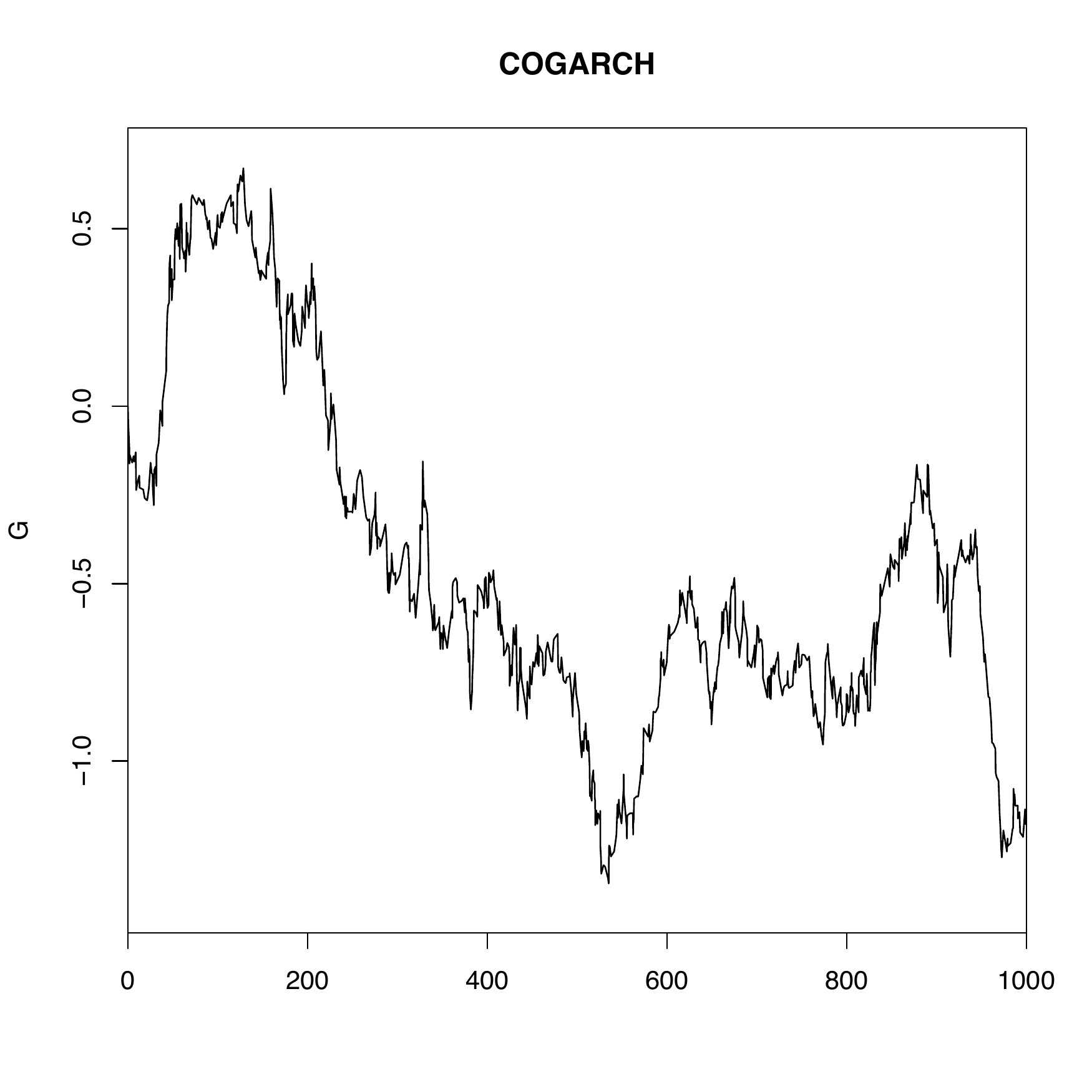} &
\includegraphics[width = 6cm]{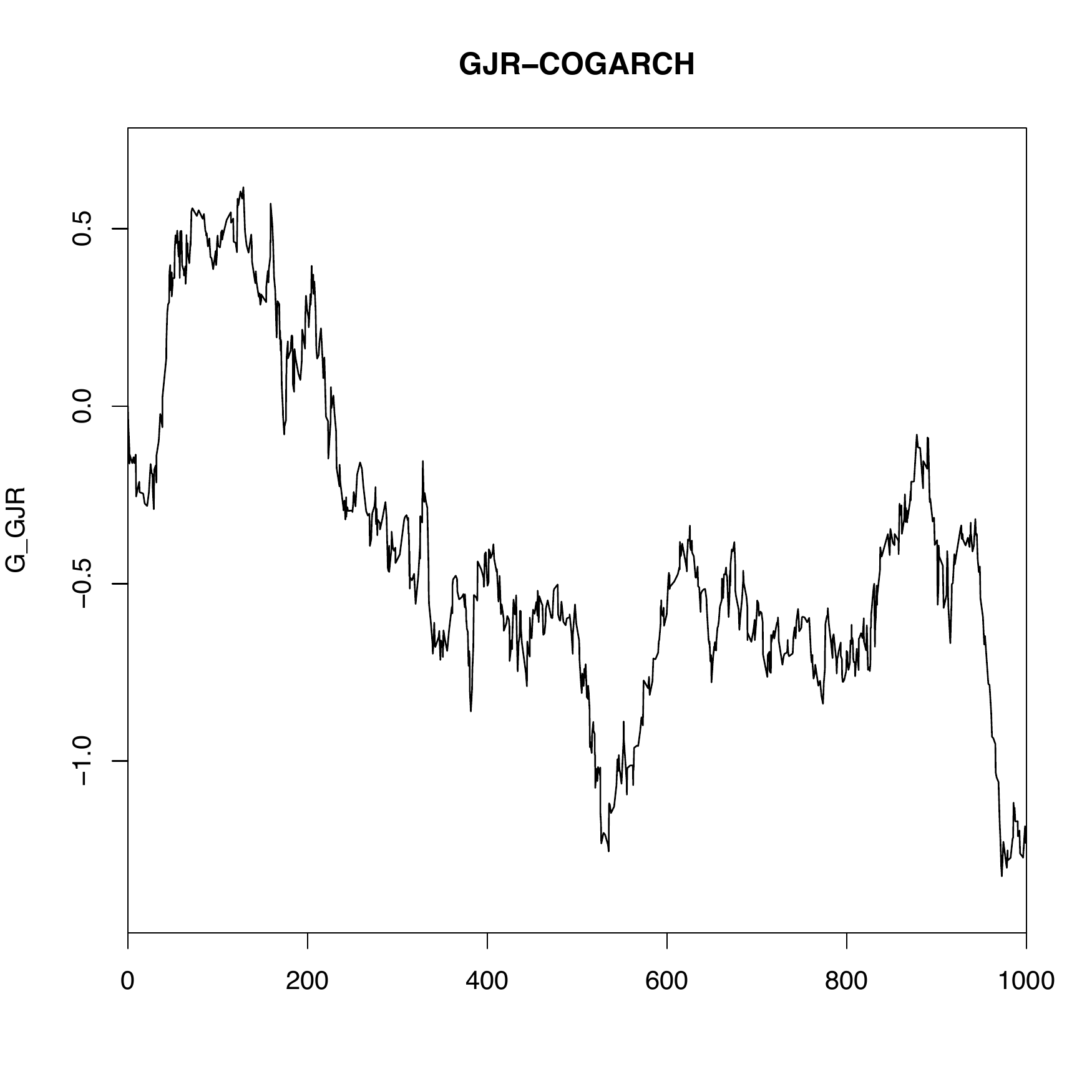}\\
\includegraphics[width = 6cm]{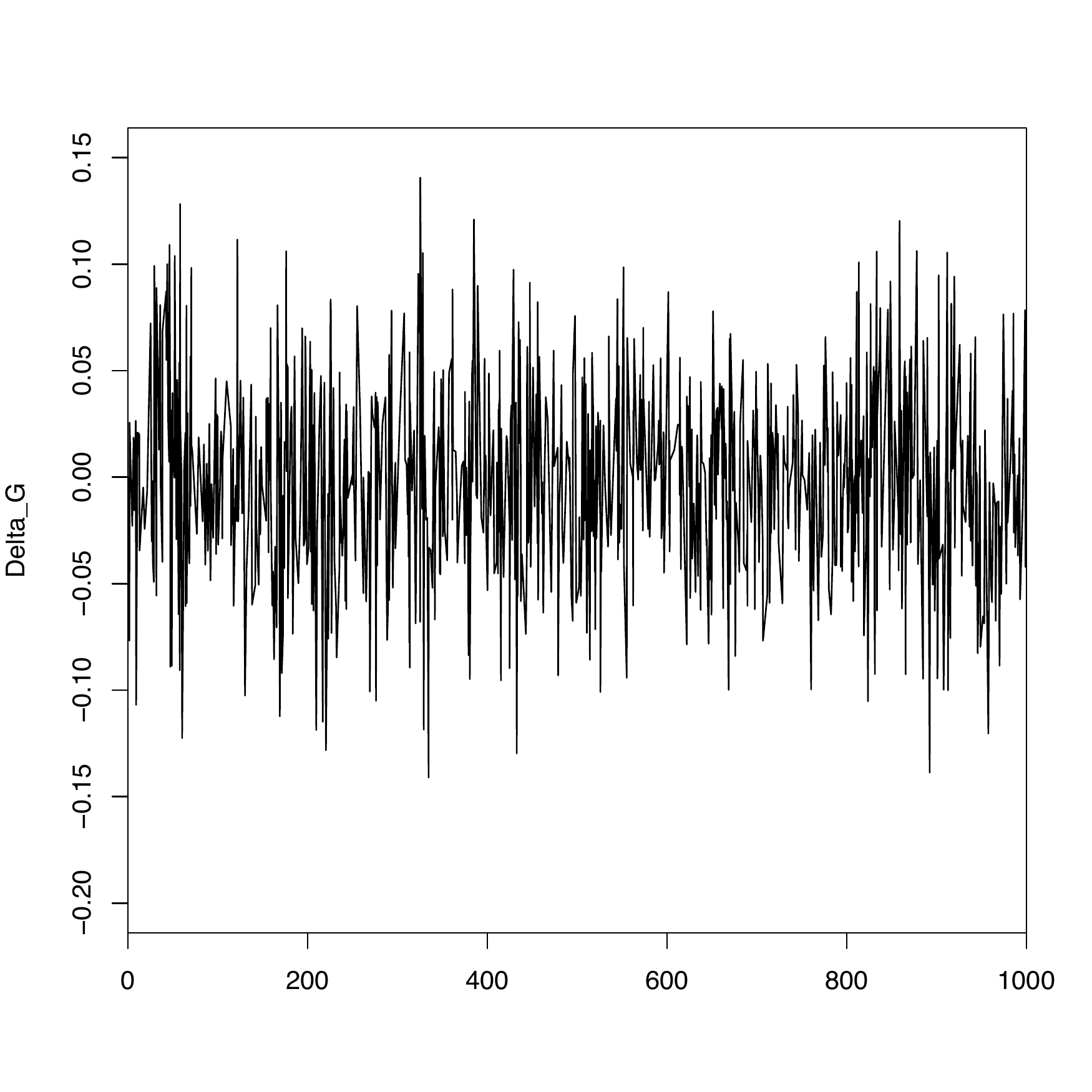} &
\includegraphics[width = 6cm]{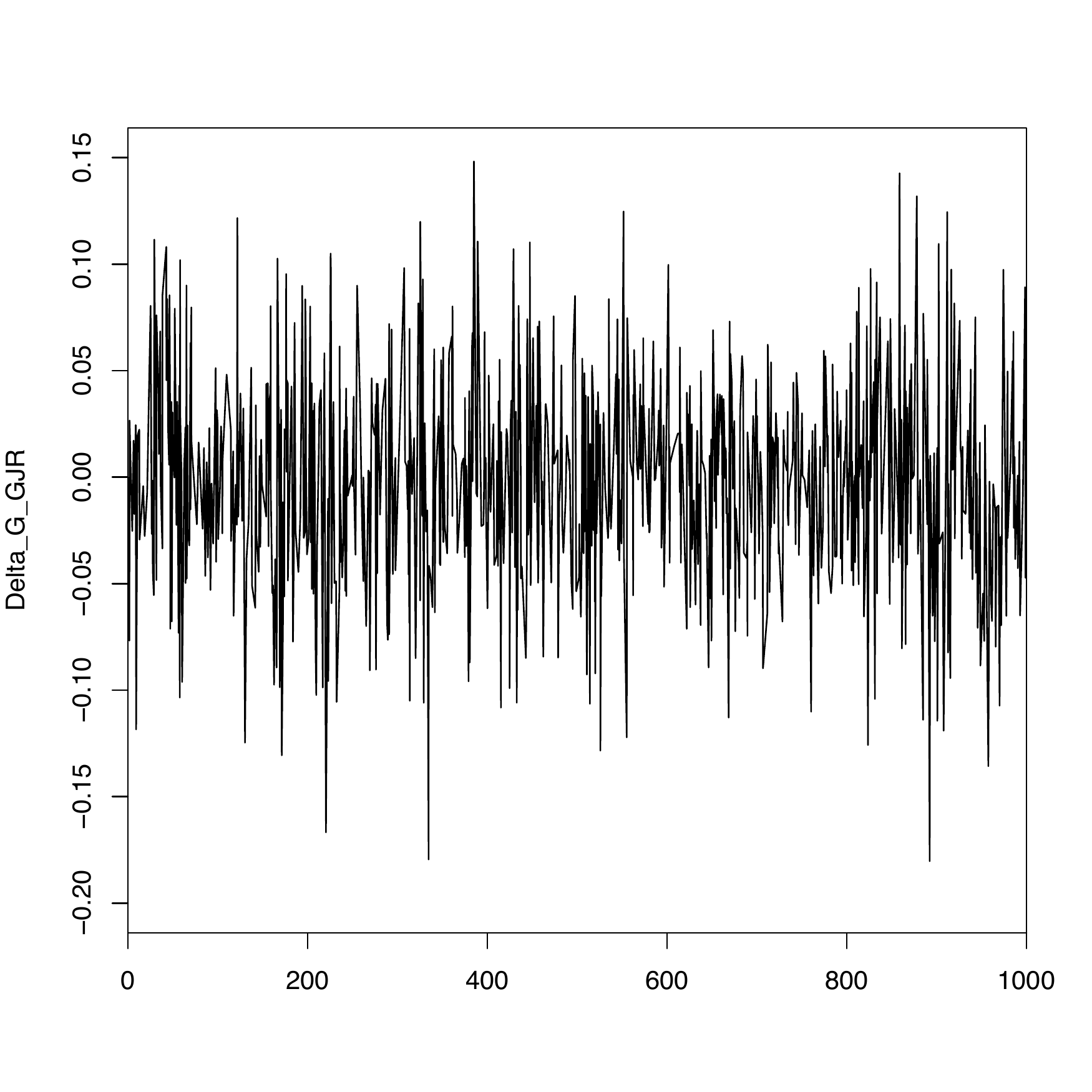}\\
\includegraphics[width = 6cm]{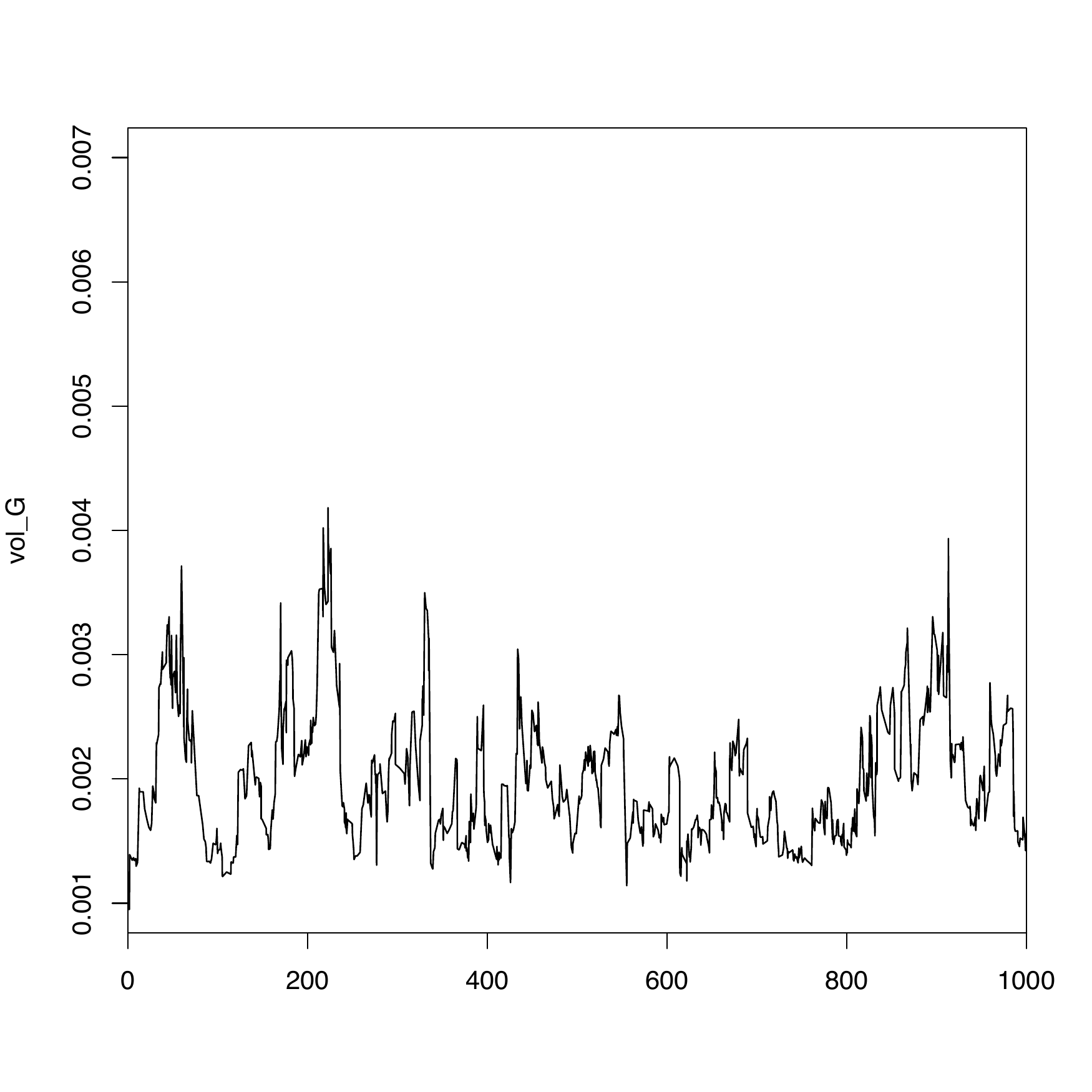} &
\includegraphics[width = 6cm]{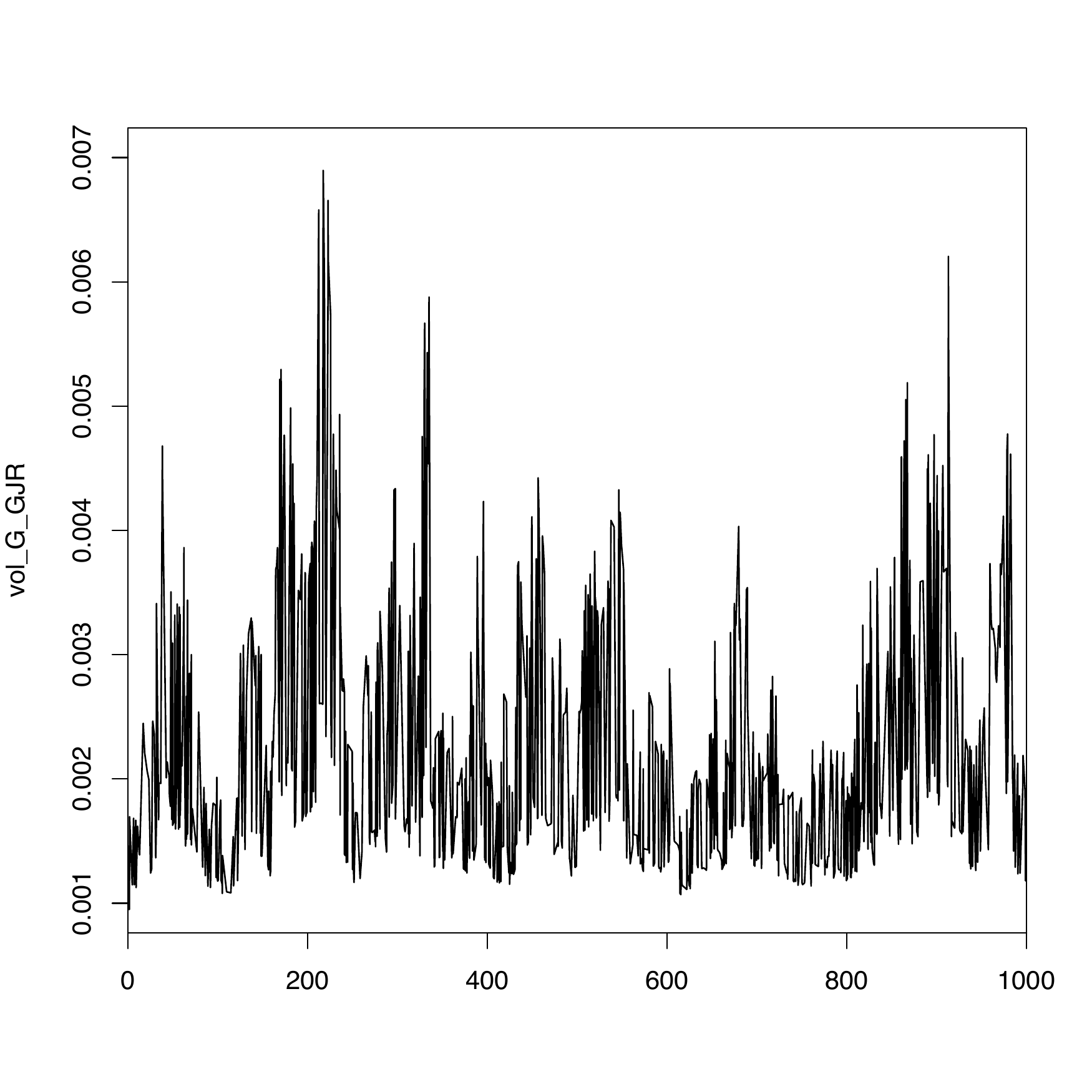}\\[-2mm]
\earr
$$
\end{center}
\caption{{\footnotesize Simulation of COGARCH (left column) vs. GJR-COGARCH (right column).
First row: price process (integrated process); second row: price differences for $\Delta=1$; third row: volatilities.
Parameters: $\theta=0.0001$, $\eta=-\log 0.9= 0.04576$, $\vp=1/18=0.05556$, $\ga=0.3$.}}
\label{fig:simulation}
\end{figure}

\section{Method of Moments for the GJR-COGARCH}\label{s3}

\subsection{Moments of the GJR-COGARCH}

In this subsection we present the theoretical second order structure of the returns of the integrated GJR-COGARCH and its squared process. These will be the basis of the method of moment estimation to be presented in Section~3.2.

In principle,  as remarked in \cite[Thm. 3.1 and Rem. 2]{lee10} - where no formulas were given - , the moments of the GJR-COGARCH can be computed analogously to those of the COGARCH as done in \cite{zappetal,KLM:2004}. 
Although the calculations are quite straightforward, they are tedious and lengthy. 
We restrict ourselves on presenting the explicit formulas in Propositions~\ref{sec:46} and \ref{prop-moments}; 
full proofs for all formulas apart from \eqref{eq:forthmoment} can be found in \cite{mayrmaster}. 

We start with the moments of the GJR-COGARCH volatility.

\begin{proposition}\label{sec:46}
Let $(\sigma_{t}^2)_{t\geq 0}$ be the stationary GJR-COGARCH volatility \eqref{eq:COGJR} with $\sigma_{0}^{2}\overset{d}= \sigma_\infty^2$ as  in \eqref{eq:coGJRstat}. 
Let $\kappa \in\NN$ be constant. Then $E\left[\sigma_{\infty}^{2\kappa}\right]<\infty$ if and only if $E\left[|L_{1}|^{2\kappa }\right]<\infty$ and $\Psi(\kappa)<0$. In particular, we have
\begin{align*}
E\left[\sigma_{t}^{2 \kappa}\right]&=\kappa !\theta^{\kappa}\prod_{l=1}^{\kappa}{\frac{1}{-\Psi(l)}}, \quad t\geq 0,\\
\cov\left(\sigma_{t}^{2},\sigma_{t+h}^{2}\right) &=\theta^{2}\left( \frac{2}{\Psi(1)\Psi(2)}-\frac{1}{\Psi(1)^2}\right)e^{h\Psi(1)},
\quad t,h\geq 0,
\end{align*}
with $\Psi(\cdot)$ as in \eqref{eq:142}.
\end{proposition}

The observations upon which our estimation will be based are the increments of the integrated GJR-COGARCH $(G_t)_{t\ge0}$. Hence we set for fixed $r>0$
\beam\label{Gincr}
G_t^{(r)}:= G_{t+r}-G_t=\int_{(t,t+r]}\sigma_{s-} dL_s, \quad t\geq 0.
\eeam
Obviously, this is a stationary process if the volatility $(\sigma^2_t)_{t\geq 0}$ is stationary.

\begin{proposition} \label{prop-moments}
Let $(L_{t})_{t\geq 0}$ be a pure-jump L\'evy process with $E[L_1]=0$ and $E[L_1^2]=1$. 
Assume that $\Psi(1)<0$ for $\Psi$ as in \eqref{eq:142}. 
Further let $\left(\sigma_{t}^{2}\right)_{t\geq 0}$ be the stationary GJR-COGARCH volatility \eqref{eq:COGJR} with $\sigma_{0}^{2}\overset{d}= \sigma_\infty^2$ as  in \eqref{eq:coGJRstat}.
Then for all $t\geq 0$ and $h\geq r>0$
\begin{align}
E[G_{t}^{(r)}] &= 0,\label{eq:688}\\
E[(G_{t}^{(r)})^{2}] &= \frac{\theta r}{|\Psi(1)|}E[L_{1}^{2}],\label{eq:689}\\
\cov(G_{t}^{(r)},G_{t+h}^{(r)}) &=0\label{eq:690}.
\end{align}
Assume further that $E\left[L_{1}^{4}\right]<\infty$ and $\Psi(2)<0$. Then $E[(G_t^{(r)})^4]<\infty$ and, if additionally $\int_\RR y^3 \nu_L(dy)=0$, we have for all $t\geq 0$ and $r>0$
\begin{align} \label{eq:forthmoment}
\lefteqn{ E[(G_t^{(r)})^4]}\\
&= 6 E[L_1^2] \frac{\theta^2}{|\Psi(1)|^2}\left(\frac{2\eta}{\vp} - (1+\gamma^2)  E[L_1^2] \right)\left(\frac{2}{|\Psi(2)|}-\frac{1}{|\Psi(1)|} \right)\left(r-\frac{1-e^{-r|\Psi(1)|}}{|\Psi(1)|} \right) \nonumber \\
&\quad + 2\frac{\theta^2}{\vp^2}\left(\frac{2}{|\Psi(2)|}-\frac{1}{|\Psi(1)|} \right) (1+6\gamma^2 +\gamma^4)^{-1} r + 3 \frac{\theta^2}{|\Psi(1)|^2}(E[L_1^2])^2 r^2, \nonumber
\end{align}
while for all $t\geq 0$ and $h\geq r>0$
\begin{align} \nonumber
 \cov((G_{t}^{(r)})^{2},(G_{t+h}^{(r)})^{2})
&= E[L_1^2] \frac{\theta^2}{|\Psi(1)|^3} \left( \frac{2\eta}{\vp}-(1+\gamma^2) E[L_1^2]\right) \left(\frac{2}{|\Psi(2)|}-\frac{1}{|\Psi(1)|} \right) \\
&\quad \times (1-e^{-r|\Psi(1)|})(e^{r|\Psi(1)|}-1) e^{-h|\Psi(1)|}>0. \label{eq:covsquared}
\end{align}
\end{proposition}

\begin{remark} 
(1) \, Setting $\ga=0$, all moment expressions reduce to those of the symmetric COGARCH in \cite[Prop.~1]{zappetal}.\\
(2) \, The asymmetry modelled by $\gamma$ in \eqref {eq:forthmoment} is also present in $\Psi(1)$ and $\Psi(2)$.\\
(3) \, Under the conditions of Proposition \ref{prop-moments} and for fixed $r>0$ the integrated GJR-COGARCH  $((G^{(r)}_{ri})^2)_{i\in\NN}$ has the autocorrelation structure of an ARMA(1,1) process (see e.g. \cite[Exercise 3.16]{brockwelldavis}). 
For the COGARCH this was also remarked in \cite[Lemma 2.1]{zappetal}.
Due to this $E[(G_t^{(r)})^4]$ cannot be deduced from \eqref{eq:covsquared}, which only holds for $h\geq r$.
\end{remark}

\subsection{Method of Moments (MoM)}


We aim at estimation of the model parameters $(\theta,\eta,\vp,\gamma)$ from a sample of equally spaced returns over time intervals of length $\Delta$. 
For $i\in\NN$ we denote the stationary increment process of the integrated GJR-COGARCH (cf. \eqref{Gincr}) by 
\begin{equation}\label{eq-defGi} 
G_i:=G^{(\Delta)}_{i\Delta}=G_{(i+1)\Delta}- G_{i\Delta}.
\end{equation}

The following is the main result of this section and relates the moments of the observed increments of the integrated GJR-COGARCH to its parameters. 

\begin{theorem} \label{thm:momentestimator}
 Let $L$ be a pure-jump L\'evy process with finite fourth moment, $E[L_1]=0$, $E[L_1^2]=1$ and L\'evy measure $\nu_L$ such that $\int x^3\nu_L(dx)=0$ and $S:= \int x^4\nu_L(dx)$ is known. 
 Assume $\Psi(2)<0$. 
 Let the stationary increment process of the integrated GJR-COGARCH with parameters $\theta,\eta, \vp$ and $\gamma$ be defined by \eqref{eq-defGi}.
 Let $\mu$, $\Gamma$, $k$ and $p$ be positive constants such that
\begin{align*}
 E[G_i^2]&= \mu, \quad \var(G_i^2)=\Gamma \\
\mathrm{cor} (G_i^2,G_{i+h}^2)&= ke^{-\Delta h p},\quad h\in\NN.
\end{align*}
Set
\begin{align*}
 M_1 &:= \Gamma - \frac{6k\Gamma}{E} \left(p\Delta-1+e^{-\Delta p}\right) - 2\mu^2, \quad \quad M_2 :=1-\frac{\mu^2 S}{\Delta M_1},\quad \quad  M_3:=\frac{\Delta k \Gamma p^2 S}{M_1 E},
 \end{align*}
 where  $E:= (1-e^{-\Delta p}) (e^{\Delta p}-1)$.
Then $M_1, M_2, M_3>0$.  Further set
\begin{align} \label{gammatilde}
\tilde{\gamma}_{1,2}&:= \frac{-M_3-4pS}{2pS-M_2}\pm \frac{\sqrt{8pSM_2^2 M_3+32 p^2 S^2 M_2^2+2pS M_2 M_3^2-8pSM_2^3}}{M_2(2pS-M_2)}\in \RR.
\end{align}
For $i=1,2$ define additionally $H_{i}:=\tilde{\gamma}_{i}^2+4\tilde{\gamma}_{i} -4$, and
\begin{align*}
M_4^{i} &:= \frac{ p^2}{\tilde{\gamma}_{i}^2} + 2\frac{\Delta  k \Gamma p^3}{\tilde{\gamma}_{i} M_1 E H_{i}},
\end{align*}
and choose the unique $\tilde{\gamma}\in\{\tilde{\gamma}_i, i=1,2 \}$ such that $M_4^i>0$ and
\begin{align} \label{gammaeindeutig}
\sqrt{M_4^i} H_i S \tilde{\gamma}_i &= -M_2 \tilde{\gamma}_i^2 + M_3 \tilde{\gamma}_i+H_i Sp.
\end{align}
Then $\tilde{\gamma}\in[1,2)$ and the parameters $\theta, \eta, \vp$ and $\gamma$ are uniquely determined by
\begin{align*}
 \theta &= \frac{p\mu}{\Delta}, \quad \quad \quad\quad \quad \,\,\,
\vp =  -\frac{p}{\tilde{\gamma}}+\sqrt{\frac{ p^2}{\tilde{\gamma}^2} + 2\frac{\Delta  k \Gamma p^3}{\tilde{\gamma} M_1 E (\tilde{\gamma}^2+4\tilde{\gamma} -4)}}, \\
\gamma& = \sqrt{\tilde{\gamma}-1}, \quad \mbox{and} \quad
\eta = p+\vp \tilde{\gamma}.
\end{align*}
\end{theorem}

\begin{proof}
 It follows readily from Proposition \ref{prop-moments} that
\begin{align}
\mu &= \frac{\theta \Delta}{|\Psi(1)|} \nonumber \\
\Gamma&=  6 \frac{\theta^2}{|\Psi(1)|^2}\left(\frac{2\eta}{\vp} - (1+\gamma^2)  \right)\left(\frac{2}{|\Psi(2)|}-\frac{1}{|\Psi(1)|} \right)\left(\Delta-\frac{1-e^{-\Delta |\Psi(1)|}}{|\Psi(1)|} \right) \label{eq:Gamma1} \\
&\quad + 2\frac{\theta^2}{\vp^2}\left(\frac{2}{|\Psi(2)|}-\frac{1}{|\Psi(1)|} \right) (1+6\gamma^2 +\gamma^4)^{-1} \Delta + 2 \frac{\theta^2}{|\Psi(1)|^2} \Delta^2 \nonumber \\
&:= \theta^2 \tilde{\Gamma} \nonumber\\
p&=|\Psi(1)| \nonumber \\
k&= \frac{\tilde{\Gamma}^{-1}}{|\Psi(1)|^3} \left( \frac{2\eta}{\vp}-(1+\gamma^2) \right) \left(\frac{2}{|\Psi(2)|}-\frac{1}{|\Psi(1)|} \right)(1-e^{-\Delta |\Psi(1)|})(e^{\Delta|\Psi(1)|}-1)  \label{eq:k1}
\end{align}
from which we immediately obtain the stated formula for $\theta$. 
Further setting $\tilde{\gamma}:= 1+\gamma^2$ we obtain the formula for $\gamma$. Also from \eqref{eq:psi1} we observe that
 $p=|\Psi(1)|=\eta -\vp \tilde{\gamma}$ which yields the given formula for $\eta$,  while by \eqref{eq:psi2}
\begin{align}\label{SPsi}
|\Psi(2)| &=-\Psi(2)=-2\Psi(1)- \vp^2(1+6\gamma^2+\gamma^4)S = 2p- \vp^2(\tilde{\gamma}^2+4\tilde{\gamma} -4)S.
\end{align}
Replacing $\theta$, $\gamma$, $|\Psi(1)|$ and $|\Psi(2)|$ in \eqref{eq:Gamma1} and \eqref{eq:k1} we hence obtain
\begin{align}
\Gamma&=  6 \frac{\mu^2 }{\Delta^2 }\left(\frac{2\eta}{\vp} - \tilde{\gamma}  \right)\left(\frac{2}{2p- \vp^2(\tilde{\gamma}^2+4\tilde{\gamma} -4)S}-\frac{1}{p} \right)\left(\Delta-\frac{1-e^{-\Delta p}}{p} \right)  \nonumber \\
&\quad + 2\frac{p^2 \mu^2}{\Delta \vp^2}\left(\frac{2}{2p- \vp^2(\tilde{\gamma}^2+4\tilde{\gamma} -4)S}-\frac{1}{p} \right) (\tilde{\gamma}^2+4\tilde{\gamma} -4)^{-1}  + 2 \mu^2  \nonumber \\
k&= \frac{\tilde{\Gamma}^{-1}}{p^3} \left( \frac{2\eta}{\vp}-\tilde{\gamma} \right) \left(\frac{2}{2p- \vp^2(\tilde{\gamma}^2+4\tilde{\gamma} -4)S}-\frac{1}{p} \right)(1-e^{-\Delta p})(e^{\Delta p}-1). \label{eq:k2}
\end{align}
Inserting the second equation into the first yields
\begin{align*}
\Gamma&=  6k \Gamma \left(p\Delta-1+e^{-\Delta p}\right)E^{-1} + 2k \Gamma \frac{p^3\Delta }{ \vp^2} \left( \frac{2\eta}{\vp}-\tilde{\gamma} \right)^{-1} (\tilde{\gamma}^2+4\tilde{\gamma} -4)^{-1} E^{-1} + 2 \mu^2
\end{align*}
and, hence, replacing also $\eta$
\begin{align*}
M_1&:= \Gamma - \frac{6k\Gamma}{E}\left(p\Delta-1+e^{-\Delta p}\right) - 2\mu^2 =  2k\Gamma  \frac{p^3 \Delta }{ \vp} \left( 2p+\vp \tilde{\gamma} \right)^{-1} (\tilde{\gamma}^2+4\tilde{\gamma} -4)^{-1}  E^{-1}
\end{align*}
i.e.
\begin{align}
\vp^2\tilde{\gamma} M_1  +  2\vp p M_1  - 2k\Gamma p^3 \Delta  (\tilde{\gamma}^2+4\tilde{\gamma} -4)^{-1}  E^{-1}=0. \label{eq:vpquadrat}
\end{align}
Note that $M_1>0$ since inserting \eqref{eq:Gamma1} and \eqref{eq:k1} into the definition of $M_1$ and using \eqref{eq:psi2} we see
\begin{align} \label{M1positiv}
M_1&=  2\frac{p^2 \mu^2}{\Delta^2 \vp^2}\left(\frac{2}{|\Psi(2)|}-\frac{1}{p} \right) (1+6\gamma^2 +\gamma^4)^{-1} \Delta
= 2\frac{\theta^2 \Delta}{|\Psi(1)||\Psi(2)|}S>0.
\end{align}
Hence by \eqref{eq:vpquadrat} it follows that
\begin{align*}
 \vp&= -\frac{p}{\tilde{\gamma}}\pm \sqrt{\frac{ p^2}{\tilde{\gamma}^2} + 2\frac{\Delta  k \Gamma p^3}{\tilde{\gamma} M_1 E (\tilde{\gamma}^2+4\tilde{\gamma} -4) } } =: -\frac{p}{\tilde{\gamma}}\pm \sqrt{M_4}
\end{align*}
As $M_1$ and $E$ are positive and $\tilde{\gamma}\geq 1$ we see that also $M_4$ is positive and, in particular,
 $\sqrt{M_4}>\frac{p}{\tilde{\gamma}}$. Since by definition $\vp>0$ this yields the given formula for $\vp$ in terms of $p,\mu,k,\Gamma$ and $\tilde{\gamma}$.\\
It remains to determine $\tilde{\gamma}$. Therefore, we restart with \eqref{eq:k2} which  via simple but lengthy algebra  leads to
\begin{align*}
0&= \vp^2\left(H S k\Gamma  +\frac{\mu^2}{\Delta^2 p^2}  \tilde{\gamma} H S E \right)  + \vp \frac{2\mu^2}{\Delta^2 p} H S E  - 2pk\Gamma.
\end{align*}
with $H:=\tilde{\gamma}^2+4\tilde{\gamma} -4$.
Inserting the obtained expression for $\vp$ this gives 
\begin{equation*}
 0=-2pk\Gamma+2\frac{H SE\mu^2}{\Delta^2 p}\left(\sqrt{M_4}-\frac{p}{\tilde{\gamma}}\right)+\left(H Sk\Gamma+\frac{\tilde{\gamma} H SE\mu^2}{\Delta^2 p^2}\right)\left(\sqrt{M_4}-\frac{p}{\tilde{\gamma}}\right)^2,
\end{equation*}
an equation which already determines $\tilde{\gamma}$. Further, reordering, inserting the expression for $M_4$ and summarizing we observe that this is equivalent to
\begin{equation*}
\sqrt{M_4}\frac{HS}{\tilde{\ga}}= -M_2+HS\frac{p}{\tilde{\ga}^2} + M_3\frac1{\tilde{\ga}}
\end{equation*}
and hence to \eqref{gammaeindeutig}.
Taking squares on both sides and inserting the expression for $M_4$ now leads to a quadratic equation whose solutions are given by \eqref{gammatilde}.
Hereby positivity of $M_3$ is obvious while positivity of $M_2$ follows via \eqref{M1positiv} since
\begin{align*}
M_2&= 1-\frac{\mu^2 S}{\Delta M_1}
 = 1-\frac{\mu^2 |\Psi(1)||\Psi(2)|}{\Delta^2 2 \theta^2}
= 1-\frac{ |\Psi(2)|}{2|\Psi(1)|  }
= 1-\frac{2p- \vp^2 H S }{2p}
= \frac{ \vp^2 H S }{2p}>0.
\end{align*}
In particular, this yields together with \eqref{M1positiv} and \eqref{eq:k2}
\begin{align*}
 M_3&=\frac{\Delta k \Gamma p^2 S}{M_1 E}=\frac{k \Gamma p^3 |\Psi(2)|}{2 \theta^2 E} =\left(\frac{2\eta}{\vp}-\tilde{\gamma}\right)\left(1-\frac{|\Psi(2)|}{2p}\right) =\left(\frac{2 p}{\vp}+\tilde{\gamma}\right)M_2\geq M_2
\end{align*}
since $\tilde{\gamma}\geq 1$. This implies that the expression under the square bracket in \eqref{gammatilde} is positive, since the first term under the bracket has a larger absolute value than the fourth term. 
In particular, \eqref{gammatilde} leads to two real-valued solutions from which $\tilde{\gamma}$ can be determined via \eqref{gammaeindeutig}.  
\end{proof}

\begin{remark}\label{remmoments}
(1) \, For the symmetric COGARCH, where $\gamma=0$, 
 Theorem~\ref{thm:momentestimator} reduces to \cite[Theorem 1]{zappetal}.\\
(2) \, Since the GJR-COGARCH volatility $(\sigma^2_t)_{t\geq 0}$ is a generalized Ornstein-Uhlenbeck process, by \cite[Prop. 3.4]{fasen}, it is exponentially $\beta$-mixing. 
For strictly stationary $(\sigma^2_t)_{t\geq 0}$ this then implies that the return process $(G_t^{(\Delta)})_{t\geq 0}$ as defined in \eqref{Gincr} is ergodic. 
Thus, by Birkhoff's ergodic theorem, strong consistency of the empirical moments and autocorrelation function follows. 
As shown in Theorem~\ref{thm:momentestimator}, the parameter vector $(\theta,\eta, \vp, \gamma)$ is a continuous function of the first two moments of the GJR-COGARCH and of the parameters $p$ and $k$ of the autocorrelation function.
Consistency of the moments hence implies consistency of the estimates for $(\theta,\eta, \vp, \gamma)$ (cf. Remark~3.2, Theorem~3 and Corollary~1 in \cite{zappetal}).\\
(3) \, Prediction based estimation methods for the COGARCH, which involve even higher order moments, are presented in \cite{BN}.\\
(4) \, Finally, we want to discuss the choice of $S=\int x^4 \nu_L(dx)$.
In principle, there exist two possibilities: the first one assumes that the driving L\'evy process is known (as done in \cite{BN}, where a simple variance gamma process was taken), the second one mimicks pseudo maximum likelihood estimation (PMLE) and assumes normality of the increments (regardless of the true, but unknown driving process). 
In a simulation study of the symmetric COGARCH, performed in \cite{BN}, the MLEs based on the true variance gamma driving L\'evy process showed a visible bias. 
The same effect has been observed and analysed for discrete time heteroscedastic models in 
\cite[Section~6.2.2]{straumann}, and exemplified for a Laplace distributed noise in \cite[Fig.~6.2]{straumann}.
On the other hand, for discrete time heteroscedastic models PMLEs lead to consistent and asymptotically normal estimators; cf. \cite[Ch.~5]{straumann}.
Based on this insight for discrete time heteroscedastic models we vote for the second option and recommend a ``pseudo MoM'' setting $E[(L_1)^4]=3$ in \eqref{eq:forthmoment} corresponding to the value for the normal distribution.
\end{remark}

\section{Pseudo Maximum Likelihood Estimation of the GJR-COGARCH}\label{s4}

In \cite{mallermuellerszimayer}, the authors presented a  first jump approximation of the COGARCH(1,1) process. In \cite{stelzerfirstjump}, this approach is further generalized to solutions of L\'evy driven stochastic differential equations. The results in \cite{mallermuellerszimayer} allow to explicitly construct a sequence of GARCH(1,1) processes converging to the COGARCH(1,1) process in probability in the Skorokhod topology. 
The benefit of this approximation is three-fold. 
Firstly, we obtain an alternative to the method of moment estimation as we can perform pseudo maximum likelihood estimation (PMLE), secondly, it makes it possible to use GARCH software for the estimation of the COGARCH parameters, and thirdly, estimation can be based on tick-by-tick data observed on a non-equidistant grid.

\subsection{First jump approximation of the GJR-COGARCH}

Recall the Skorokhod $J_{1}$-distance on the space $\mathbb{D}^{d}[0,T]$ of $\RR^d$-valued, c\`adl\`ag functions, indexed by $[0,T]\subset \RR_+$ given by
\begin{equation}
\rho_d(U,V)=\inf_{\lambda\in\Lambda}{\left\{\sup_{0\leq t\leq T}{\|U_{t}-V_{\lambda(t)}\|+\sup_{0\leq t\leq T}{|\lambda(t)-t|}}\right\}},
\end{equation}
for two processes $U$ and $V$ in $\mathbb{D}^{d}[0,T]$, where $\Lambda$ is the set of all  increasing, continuous functions with $\lambda(0)=0$ and $\lambda(T)=T$.

Now let $0=t_{0}(n)<t_{1}(n)<\dots<t_{N_{n}}(n)=T$ be a sequence of partitions of the time intervall $[0,T]$ such that $\lim_{n\to\infty}{N_{n}}=\infty$ and $\Delta t(n):=\max_{i=1,\dotsc, N_{n}}{\Delta t_{i}(n)}\to 0$ as $n\to \infty$, where $\Delta t_{i}(n):=t_{i}(n)-t_{i-1}(n)$. For each $n\in\NN$ we define the discrete-time processes $(G_{i,n})_{i=1,\dotsc, N_{n}}$ and $(\sigma_{i,n}^{2})_{i=1,\dotsc, N_{n}}$ recursively via
\begin{align}\label{eq:GJRapprox}
G_{i,n}&=G_{i-1,n}+\sigma_{i-1,n}\sqrt{\Delta t_{i}(n)}\varepsilon_{i,n},\quad i=1,2,\dots,N_{n}, \\ 
\sigma_{i,n}^{2}&=\theta\Delta t_{i}(n) \label{eq:GJRapprox2}\\
& +\left(1+\left[(1-\gamma)^{2}\mathds{1}_{\{\varepsilon_{i-1,n}> 0\}}+(1+\gamma)^{2}\mathds{1}_{\{\varepsilon_{i-1,n}<0\}}\right]\varphi\Delta t_{i}(n)\varepsilon_{i-1,n}^{2}\right)e^{-\eta\Delta t_{i}(n)}\sigma_{i-1,n}^{2} \nonumber
\end{align}
with $G_{0,n}=G(0)=0$. 
The innovations sequences $(\varepsilon_{i,n})_{i=1,\ldots, N_n}$ for $n\in\NN$ 
are constructed via a first-jump approximation of the driving L\'evy process $L$.

Observe that by setting $Y_{i,n}:=G_{i,n}-G_{i-1,n}$ Eq.~\eqref{eq:GJRapprox2} is equivalent to
$$\sigma_{i,n}^{2}= \theta\Delta t_{i}(n)+e^{-\eta\Delta t_{i}(n)}\sigma_{i-1,n}^{2}+
\varphi e^{-\eta\Delta t_{i}(n)} (|Y_{i,n}|-\ga Y_{i,n})^2. $$
Hence \eqref{eq:GJRapprox} and \eqref{eq:GJRapprox2} describe a recursion of a GJR-GARCH process. 
In particular, for equidistant time steps a reparametrisation yields equivalence of \eqref{eq:GJRapprox} and \eqref{eq:GJRapprox2} to  \eqref{eq:GJR}. 

To construct the innovations $(\varepsilon_{i,n})_{i=1,\ldots, N_n}$ in \eqref{eq:GJRapprox} let $(m(n))_{n\in\NN}$ be a positive, decreasing sequence converging to $0$, which is bounded above by $1$. 
Assume that
$\lim_{n\to \infty} \Delta t_{i}(n) (\nu_L(\{ |x|\geq m(n) \}))^2=0$, and define for all $n\in\NN$
$$\tau_{i,n}:=\inf\{t:t_{i-1}(n)<t\leq t_i(n), |\Delta L_t|>m(n) \} \quad \mbox{for all } i=1,\ldots, N_n,$$
while $\tau_{i,n}:=+\infty$ if $L$ has no jump larger than $m(n)$ in the interval  $(t_{i-1}(n),t_{i}(n)]$. 
Then we define
\begin{equation}
\varepsilon_{i,n}=\frac{\mathds{1}_{\{\tau_{i,n}<\infty\}}\Delta L_{\tau_{i,n}}-\mu_{i}(n)}{\xi_{i}(n)}, \quad i=1,2,\dots,N_{n},
\end{equation}
where $\mu_{i}(n)$ and $\xi_{i}^2(n)$ denote the (finite) expectation and variance of the i.i.d. random variables $(\mathds{1}_{\{\tau_{i,n}<\infty\}}\Delta L_{\tau_{i,n}})_{i=1,\dots,N_{n}}$.

The discrete time processes $\sigma_{.,n}^{2}$ and $G_{.,n}$ as in \eqref{eq:GJRapprox} can then be embedded in a continuous time setting by taking $(\sigma_{n}^{2}(t))_{t\geq 0}$ and $(G_{n}(t))_{t\geq 0}$ as
\begin{equation}
\sigma_{n}^{2}(t):=\sigma_{i,n}^{2}\quad \text{and}\quad G_{n}(t):=G_{i,n} , \quad \mbox{ for all }t\in[t_{i-1}(n),t_{i}(n)), \ 0\leq t\leq T,\label{eq:213}
\end{equation}
with $G_{n}(0)=0$.

Now we can formulate the main result of this section.

\begin{theorem}\label{sec:200}
Define the bivariate processes $(\sigma^{2},G)$ as in \eqref{eq:defCOGJR} and \eqref{eq:defCOGJRint} and $(\sigma_{n}^{2},G_{n})_{n\in\NN}$ by \eqref{eq:213}. 
Then
\begin{equation}
\lim_{n\to\infty} \rho_2 \left((\sigma_{n}^{2},G_{n}),(\sigma^{2},G)\right)= 0\, \mbox{ in probability}.
\end{equation}
\end{theorem}

\begin{proof}
The long and technical proof of Theorem \ref{sec:200} can be carried out along the lines of the proof of \cite[Thm. 2.1]{mallermuellerszimayer}, replacing $(\Delta L_{\tau_{i,n}})^2$ by
$h(\Delta L_{\tau_{i,n}})$ with $h$ as in \eqref{eq:defCOGJR}; for details see \cite{mayrmaster}.
\end{proof}

\subsection{Pseudo maximum likelihood estimation (PMLE)}

In this section we extend the pseudo maximum likelihood estimation (PMLE) method for the COGARCH from Maller et al. \cite{mallermuellerszimayer}. 
As in the MoM, we aim at estimation of the model parameters $(\theta,\eta,\vp,\gamma)$ where, other than for the MoM, we allow for unequally spaced returns as observations.  
The basic idea is to replace the unknown likelihood by a corresponding normal likelihood; in our case we will assume that the increments of the integrated GJR-COGARCH are normally distributed. 

Assume we are given observations $G_{t_{i}}$ of the integrated GJR-COGARCH as in \eqref{eq:defCOGJRint} and \eqref{eq:defCOGJR} at fixed (non-random) times $0=t_{0}<t_{1}<\cdots <t_{N}=T$.
We denote the observed returns by $Y_{i}:=G_{t_{i}}-G_{t_{i-1}}$ and the time-steps by $\Delta t_{i}:=t_{i}-t_{i-1}$ for $i=1,\ldots,N$.
We also assume to be in the stationary regime.
Then
\begin{equation}
Y_{i}=\int_{t_{i-1}}^{t_{i}}\sigma_{s-}dL_s,
\end{equation}
for a L\'evy process $L$ with $E[L_1]=0$ and $E[L^{2}_1]=1$.
Denote by $\mathcal{F}_{t_{i-1}}$ the sigma algebra generated by $\{Y_{k}: k\le i-1\}$, then
the returns $Y_{i}$ are conditionally independent of $Y_{i-1},Y_{i-2},\dotsc$, given $\mathcal{F}_{t_{i-1}}$, since $(\sigma_{t}^{2})_{t\geq 0}$ is a Markov process (e.g. \cite[Lemma 3.3]{behmelindnermaller}).
In particular, by independence of the L\'evy increments, we have
\begin{equation*}
E[Y_{i}|\mathcal{F}_{t_{i-1}}]=E[Y_{i}]=E[G_{t_{i}}-G_{t_{i-1}}]=0,
\end{equation*}
while one deduces similarly as in the proof of Eq. (5.4) in \cite{KLM:2004},
\begin{align}\label{eq:1101}
\rho_{i}^{2}:= E[Y_{i}^{2}|\mathcal{F}_{t_{i-1}}]
&= E[L_1^2]  \left((\sigma^{2}_{t_{i-1}}-E[\sigma^{2}_0])
\frac{e^{-\Delta t_{i}\Psi(1)}-1}{-\Psi(1)}+E[\sigma^2_0]\Delta t_{i}\right),
\end{align}
with $\Psi(1)$ as in \eqref{eq:psi1}. 
Moreover, due to the stationarity assumption and Proposition~\ref{sec:46} we set
$$E[\sigma^{2}_0]=\frac{\theta}{-\Psi(1)}=\frac{\theta}{\eta-\varphi(1+\gamma^{2})}.$$
Inserting this in \eqref{eq:1101} yields
\begin{align}\label{eq:1515}
\rho_{i}^{2}
&=\left(\sigma^{2}_{t_{i-1}}-\frac{\theta}{\eta-\varphi(1+\gamma^{2})}\right)
\frac{e^{\Delta t_{i}(\eta-\varphi(1+\gamma^{2}))}-1}{\eta-\varphi(1+\gamma^{2})}
+\frac{\theta}{\eta-\varphi(1+\gamma^{2})}\Delta t_{i}.
\end{align}

To apply PMLE we assume that the returns $Y_i$ are conditionally normal distributed with expectation $0$ and variance $\rho_{i}^{2}$ given in \eqref{eq:1515}. The occuring sequence $(\sigma_{t_i}^2)_{i=1,\ldots, N_n}$ can be iterated starting from  $\si_0:=E[\si_0] = \theta/(\eta-\vp(1+\ga^2))$ and using the observations $Y_0,\ldots,Y_N$ via the first jump approximation model \eqref{eq:GJRapprox} and \eqref{eq:GJRapprox2}, i.e. 
\begin{align}
Y_i &=  \sigma_{i-1}\sqrt{\Delta t_{i}}\varepsilon_{i-1}\\
\sigma_{t_i}^{2}
&=\theta\Delta t_{i}+ e^{-\eta\Delta t_{i}}\sigma_{t_{i-1}}^{2}+
\varphi e^{-\eta\Delta t_{i}} (|Y_{i-1}|-\ga Y_{i-1})^2 \label{eq:1516}
\end{align}
which is of the form \eqref{eq:GJR} 
with parameters $\theta\Delta t_{i}$,
$\al= \varphi e^{-\eta\Delta t_{i}}$ and $\beta =e^{-\eta\Delta t_{i}}$.

Then we obtain as PML function
\begin{align}
\mathcal{L}_{N}=\mathcal{L}_{N}(\theta,\varphi,\eta,\gamma)
&=\log\left(\prod_{i=1}^{N}\frac{1}{\rho_{i}\sqrt{2\pi}}e^{-\frac{1}{2}\left(\frac{Y_{i}}{\rho_{i}}\right)^{2}}\right)\nonumber\\
&=-\frac{1}{2}\sum_{i=1}^{N}\log(\rho_{i}^{2})-\frac{N}{2}\log(2\pi)-\frac{1}{2}\sum_{i=1}^{N} \frac{Y_{i}^{2}}{\rho_{i}^{2}}.\label{eq:1600}
\end{align}
 Now one can use standard algorithms to obtain the pseudo maximum likelihood estimators as
$${\rm argmin}_{\theta,\varphi,\eta,\gamma} \sum_{i=1}^{N}\left(\frac{Y_{i}^{2}}{\rho_{i}^{2}} - \log(\rho_{i}^{2})\right),$$
or special algorithms designed for GARCH and GJR GARCH models (cf. \cite{bollerslev,dingetal,PenzerWangYao} or \cite[Chapter 5]{straumann}).

\brem
Consistency and asymptotic normality under certain regularity conditions have been proved
for the GARCH and the asymmetric GARCH model in \cite{MS}, and for the symmetric COGARCH in \cite{KL}.
Consequently, such results are also expected to hold for the GJR-COGARCH.
\erem

\section{Conclusion}

We extend the COGARCH(1,1) model to an asymmetric GRJ-COGARCH(1,1), which allows us to capture the observed asymmetry in financial data. Under stationarity conditions we calculate up to four moments and the covariance function of the squared returns of the integrated process.  Matching the analytical and empirical moments for the GJR-COGARCH is by no means standard and involves complex calculations as they are given in Section \ref{s3}.2. This method needs equidistant data.

We also derive the first jump approximation of the GRJ-COGARCH and prove convergence in probability in the Skorohod topology.
The PMLE based on normality of the returns is derived in Section \ref{s4}.2 and is the basis for the use of algorithms developed for the discrete time GJR-GARCH. This method has the advantage also to apply to irregularly spaced data.

\section{Acknowledgements}
We thank Stephan Haug for programming support and Thorsten Kud for the simulations depicted in Figure \ref{fig:simulation}.

\end{document}